\documentclass[12pt,reqno]{amsart}
\usepackage{latexsym, amsmath, amscd, amssymb, amsthm}
\usepackage{bm}
\usepackage{mathrsfs}   
\usepackage{xspace}
\usepackage{enumerate}
\usepackage{verbatim}
\usepackage[english]{babel}
\usepackage{epsfig}
\usepackage[section]{algorithm}
\usepackage[noend]{algpseudocode}

\begin{document}
\newtheorem{lemma}{Lemma}[section]
\newtheorem{theorem}[lemma]{Theorem}
\newtheorem{corollary}[lemma]{Corollary}
\newtheorem{proposition}[lemma]{Proposition}
\theoremstyle{definition}
\newtheorem{definition}[lemma]{Definition}
\newtheorem{conjecture}[lemma]{Conjecture}
\newtheorem{question}[lemma]{Question}
\newtheorem{problem}[lemma]{Problem}
\newtheorem{claim}[lemma]{Claim}
\newtheorem{example}[lemma]{Example}
\newtheorem{remark}[lemma]{Remark}
\newtheorem{assumption}[lemma]{Assumption}
\newtheorem*{acknowledgements}{Acknowledgments}
\theoremstyle{remark}

\newcommand{\cf}{{\em cf.}\xspace }
\newcommand{\eg}{{\em e.g.},\xspace }
\newcommand{\ie}{{\em i.e.},\xspace }
\newcommand{\etal}{{\em et al.}\ }
\newcommand{\etc}{{\em etc.}\@\xspace}

 \newcommand{\R}{{\mathbb R}}
 \newcommand{\Z}{{\mathbb Z}}
 \newcommand{\T}{{\mathbb T}}
 \newcommand{\C}{{\mathbb C}}
 \newcommand{\Q}{{\mathbb Q}}
 \newcommand{\N}{{\mathbb N}}

 \newcommand{\alli}{{i=1, \ldots, t  }}
 \newcommand{\allj}{{j=0, \ldots, e_{i}-1}}

 
\title[Polynomial representability]{A faster algorithm for testing polynomial
  representability of functions over finite integer rings}  
\author[A. Guha]{Ashwin Guha}
\address{Department of Computer Science and Automation \\ Indian Institute of
Science \\ Bangalore 560012, India.}
\email{guha\_ashwin@csa.iisc.ernet.in}

\author[A. Dukkipati]{Ambedkar Dukkipati}
\address{Department of Computer Science and Automation \\ Indian Institute of
Science \\ Bangalore 560012, India.}
\email{ad@csa.iisc.ernet.in}

\begin{abstract}
Given a function from $\Z_n$  to itself one can determine its
polynomial representability by  using Kempner function.  
In this paper we present an alternative characterization of
polynomial functions over $\Z_n$ by constructing a generating set for
the $\Z_{n}$-module of polynomial functions. This characterization
results in an algorithm that is  faster on average in deciding
polynomial representability. We  also extend the characterization to
functions in several variables.     
\end{abstract}

\maketitle

\section{Introduction}
\label{Intro}
\noindent
In this paper we deal with the following question: 
given a function from a finite integer ring to itself does 
there exist a polynomial that evaluates to the  function?
In the case of real numbers $\R$, if the function is specified at only a finite
number of points it is possible to obtain a  
polynomial using Lagrange interpolation~\cite{lagrangereflexions}. For
analytic functions one may get an approximation using Taylor's
series. This problem has been well-studied over finite fields as
well. It was noted by Hermite \cite{hermite1863fonctions} that every
function over finite field of the form $\Z_p$, which is the set of
integers  modulo prime $p$, can be represented by a polynomial. This
result was extended by Dickson \cite{dickson1896analytic} to finite
fields $F_q$, where 
$q$ is a prime power. Moreover, it was also shown that there exists a
unique polynomial of degree less than $q$ that evaluates to the given
function. A thorough study of finite fields can be found in
\cite{lidl1997finite}. 

The property of polynomial representability does not hold
over finite commutative rings. In this paper we study the problem of
polynomial representability over finite integer rings
${\mathbb{Z}}_{n}$, which is the set of residue classes of $\Z$ modulo $n$. 

The earliest work in this direction was by Kempner
\cite{kempner1921polynomials}. It was proved that the only  
residue class rings over which all functions can be represented by polynomials 
are $\Z_p$, where $p$ is prime. Kempner
\cite{kempner1921polynomials} also introduced the function (sometimes
referred to as Smarandache function) defined as follows. 
\begin{definition}
\label{mufunction}
Kempner function $\mu: \N \longrightarrow \N$ is defined as $\mu(n)$ is the smallest positive integer such that $n \:|\: \mu(n)!$. 
\end{definition}
The Kempner function plays an important role in the study of polynomial
functions. In his work, Kempner showed that there exists a   
polynomial of degree less than $\mu(n)$ that evaluates to a function
over $\Z_n$, if the function is polynomially representable. An easy 
method to calculate $ \mu(n)$ is also given in
\cite{kempner1921polynomials}. One can show that when $n$ factors into primes as  $p_1^{e_1}p_2^{e_2} \ldots p_t^{e_t}$, 
then $\mu(n) = \max(\mu(p_i^{e_i})) $  is of the form $r \cdot p_k$ for some prime divisor 
$p_k$ of $n$ where $r$ is a positive integer less than or equal to $e_k$.  

It is obvious that the Kempner function is not monotonic: when $n$ is
prime $\mu(n)=n$, otherwise $\mu(n) < n$.  Kempner function has been
studied for its own merit and a discussion on the properties of this function is 
beyond the scope of this paper.  However, one may claim that as $n$ increases, $\mu(n)$
tends to be much smaller than $n$, by which one means that for most cases
$\mu(n)$ tends to be sub-logarithmic compared to $n$ \cite{luca2001average}. 
 
Polynomial representation in $\Z_n$ has since then been studied by
Carlitz \cite{carlitz1964functions}.  
The number of polynomial functions over $\Z_n$, when $n$ is a prime power, is given by Keller and Olson \cite{keller1968counting}. This was extended to arbitrary positive integer $n$ by Singmaster~\cite{singmaster1974polynomial}, where the Kempner function was used to give a  canonical representation for the polynomial functions.  Other notable results are given in \cite{mullen1984polynomial,brawley1992functions,chen1995polynomial,chen1996polynomial}.

Recently, the problem of polynomial representability of functions in
several variables has 
been studied by Hungerb\"uhler and Specker ~\cite{hungerbuhler2006generalization}. In this work, an elegant 
characterization of polynomial functions was given by generalizing the 
Kempner function to several variables.  The result makes use of
partial difference operator to determine whether  
a given function from $\Z_n^m$ to $\Z_n$ is polynomially
representable. This work does not provide a
computational complexity analysis but one can see that this method does
not lead to an efficient algorithm for verifying polynomial
representability of the functions. The characterization involves 
repeated computation of the  difference operator leading to an
algorithm whose time complexity is very large. In terms of computation, its   
performance is comparable to the intuitive method of checking for
existence of scalars $c_0, \ldots, c_{\mu(n)-1} \in \Z_n$ such that
the polynomial  
$\sum_{i=0}^{\mu(n)-1} c_i X^i$ evaluates to the given function. For
instance, in  the case of single variable, the computation of
$\Delta^k g(0) $ requires $O(k)$ operations for each $0 \leq k \leq
n$, hence checking polynomial representability may require $O(n^2)$
operations.  

In this paper we present a new characterization by adopting an
entirely new approach that gives rise to a faster algorithm. For this, we
generalize a characterization of polynomial functions over $\Z_{p^e}$
that is proposed in~\cite{guhaalgorithmic}.

\subsection*{Contributions}
In this paper we give an alternative characterization of polynomial functions over $\Z_n$. 
The new characterization is based on the fact that the set of polynomial functions forms a $\Z_n$-submodule of the $\Z_n$-module 
of all functions from $\Z_n$ to itself. We describe a `special' generating set for this $\Z_n$-module of polynomial functions. 
When $n$ is prime this generating set forms the standard basis for the vector space of polynomial functions.
We present a new algorithm based on this characterization and show that this is faster on average in deciding the polynomial 
representability of functions. We also extend the characterization to
functions in several variables and present a analysis of the algorithm
in this case. 

\subsection*{Organization}
The paper is organized as follows. Section \ref{Bkgd} contains the notation and necessary basic lemmas. The main theorem and
the characterization are given in Section \ref{Charac}. An algorithm based on the result is given in Section \ref{MyAlgo}.
In Section \ref{Algocomplexity} we discuss the complexity of the algorithm and compare its performance against an algorithm 
that makes use of canonical set of generators.  
The result is extended to functions in several variables in Section \ref{Multivar}. Section \ref{coda} contains the concluding remarks.

\section{Background}
\label{Bkgd}
\noindent
Throughout this paper we use $n$ to denote a positive integer of the form $ n= p_1^{e_1}p_2^{e_2} \ldots p_t^{e_t}$, where 
$ p_1 < p_2 < \ldots < p_t $ are distinct primes. Kempner function is denoted by $\mu$ (Definition \ref{mufunction}). Since $n$ is fixed through out this paper, we abbreviate  $\mu(n)$ to $\mu$ in some formulae. 
In $\Z_n $, each element of the congruence class is represented by the least non-negative residue modulo $n$ and 
all computations are performed modulo $n$ unless explicitly mentioned  otherwise. Polynomials are of the form $c_0 + c_1 X + \ldots + c_r X^r$, where $X$ is the indeterminate and  coefficients are from $\Z_n$.

A function $f : \Z_n \longrightarrow \Z_n$ is represented as an $n$-tuple $(a_0, a_1, \ldots, a_{n-1})$, where the $i$\textsuperscript{th} component $a_i = f(i)$, for $i=0, \ldots, n-1$. Hence we denote the set of all functions by $\Z_{n}^{n}$.

Given $v= (a_0, a_1, \ldots, a_{n-1}) $, $v^{<k>}$ represents the $k$\textsuperscript{th} cyclic shift to the right, 
for $k=0, \ldots, n-1$. That is  
\begin{displaymath}
 v^{<k>} = (a_{n-k}, a_{n-k+1}, \ldots, a_{n-k-1}),
\end{displaymath}
and we assume that $v^{<0>} =v$. 
In other words $v^{<k>}(i) = v(i-k)$ for all  $k=0, \ldots, n-1$.

Given a set $\{ v_1, v_2, \ldots, v_r \} \subset \Z_n^n$, $\langle v_1, v_2, \ldots, v_r \rangle$ denotes the $\Z_n$-module generated by that set. 
$\langle\langle v_1, v_2, \ldots, v_r \rangle\rangle $ denotes the $\Z_n$-module
generated by $v_i$ with $i=1, \ldots, r $ along with their cyclic shifts, \ie 
$ \langle \langle v_1, v_2, \ldots, v_r \rangle \rangle = \{  \sum \alpha_{ij}v_{i}^{<j>}  \:|\: \alpha_{ij} \in \Z_n, i=1, \ldots, r $, $j=0 , \ldots, n-1    \} $. 
We say a function is polynomial if there exists some polynomial in $\Z_n[X]$ that evaluates to the given function. 

We now make a few simple observations that are easy to verify. 

\begin{lemma}
\label{cyclpolyn}
 Suppose $v \in \Z_n^n$ is a polynomial function. Then $v^{<k>}$ is also a polynomial function for all $k=0, \ldots, n-1$.
\end{lemma}
This is easy to see since if $f(X) \in \Z_n[X]$ evaluates to $v$, then $f(X-k)$ which is also a polynomial evaluates to $v^{<k>}$.

\begin{lemma}
 Suppose $u,v \in \Z_n^n$ are polynomial functions. Then $\alpha u + \beta v$ is also a polynomial function for all $ \alpha, \beta \in \Z_n$. In other words, the set of all polynomial functions forms a $\Z_n$-module.
\end{lemma}
This is also obvious since if $f(X)$ and $g(X) \in \Z_n[X]$ evaluate to $u$ and $v$ respectively then $\alpha f + \beta g$ is also polynomial that evaluates to $\alpha u + \beta v$.

\begin{lemma}
 Suppose $u,v \in \Z_n^n$ are polynomial functions. Then $ u \cdot v$ defined by componentwise multiplication, 
$$ (u \cdot v) (x) = u(x) \cdot v(x), \mbox{ for all } x \in \Z_n$$
is also a polynomial function.
\end{lemma}
This is simply the assertion that if $f(X)$ and $g(X) $ are polynomials then $f(X)g(X)$ is also a polynomial. This lemma states that the polynomial functions form a $\Z_n$-algebra.  Our objective is to provide a set of generators that 
generate the set of polynomial functions as a $\Z_n$-module. In particular we look for a set $S$ such that 
$ \langle\langle v \:|\: v \in S \rangle\rangle$ is the set of polynomial functions.
\begin{definition}
 We call a set $S \subset \Z_n^n$, $\Z_n$-multiplicatively-closed if for all $u, v \in S$,  $ u \cdot v = \alpha w$ for some $\alpha \in \Z_n $ , $w \in S$, where $\alpha$ may be zero.
\end{definition}

This definition is similar to that of closure for any binary operation except that we allow the product to be a scalar multiple of an element in the set. Such a definition 
ensures that the module generated by such a set $S$ is equal to the algebra generated by $S$. We now state the first non-trivial yet simple lemma.

\begin{lemma}
\label{contlemma}
 Let $S \subset \Z_n^n$ be a $\Z_n$-multiplicatively-closed set. If the functions corresponding to $1$ and $X$ belong to the module generated by $S$, then every polynomial function belongs to the module generated by $S$.
\end{lemma}

\begin{proof}
 Let the functions induced by polynomials $1$ and $X$ belong to $\langle S \rangle$. This means the vectors $ (1,1, \ldots,1)$ and $(0,1,2, \ldots,n-1)$ lie in $\langle S \rangle$.
 It suffices to show that evaluations of $1, X, X^2, \ldots, X^{\mu-1}$ lie in $ \langle S \rangle $, since any polynomial function can be represented by a unique polynomial of degree less than $ \mu$.
Let  $X = \displaystyle\sum\limits_{i=1}^{k} a_iu_i,$ where $ a_i \in \Z_n, u_i \in S.  $

\begin{displaymath} 
 X^2 = X\cdot X = ( \displaystyle\sum\limits_{i=1}^{k} a_iu_i)\cdot (\displaystyle\sum\limits_{i=1}^{k}  a_iu_i) 
     = \displaystyle\sum\limits_{1 \leq i,j \leq k} a_i a_j u_i\cdot u_j = \displaystyle\sum\limits_{i=1}^{l} c_i  v_i,
\end{displaymath}
where $c_i \in \Z_n, v_i \in S$. Hence, function induced by $X^2$ belongs to $\langle S \rangle$. Similarly one can show that all exponents of $X$ lie in $\langle S \rangle$.
\end{proof}

\section{Characterization}
\label{Charac}
\noindent
Let $n= p_1^{e_1}p_2^{e_2} \ldots p_t^{e_t}$. Consider the functions $u_{p_i,j}: \Z_n \longrightarrow \Z_n$ defined as follows.
\footnote{In the function definition, $a$ is the canonical representative of a congruence class. 
$p_i \:|\: a $ means $p_i$ divides $a$ as integers.}
\begin{equation}
\label{upijs}
 u_{p_i,j}(a)=	\left\{
			\begin{array}{ll}
			     \frac{\displaystyle {n}}{\displaystyle{p_i^{e_i}}}a^j  & \mbox{ if } p_i \:|\: a, \\
			      0 & \mbox{ if } p_i \nmid a,
			\end{array}
\right.
\end{equation}
for all $ \alli$, $\allj$.
In vector notation, when $j \neq 0$,
$$ u_{p_i,j} = \frac{n}{p_i^{e_i}}(0,\ldots, (p_i)^j, \ldots, (2p_i)^j, \ldots, (n-p_i)^j, 0, \ldots, 0), $$
where entries for non-multiples of $p_i$ are zero. The cyclic shifts of $u_{p_i,j}$ are defined as

\begin{equation*}
 u_{p_i,j}^{<k>}(a)=	\left\{
			\begin{array}{ll}
			    \frac{\displaystyle {n}}{\displaystyle{p_i^{e_i}}}(a-k)^j  & \mbox{ if }  a \equiv k \; (\mbox{mod }p_i), \\
			      0 & \mbox{ otherwise, } 
			\end{array}
\right.
\end{equation*}
which corresponds to $u_{p_i,j} $ shifted by $k$ places to the right for $\alli, \allj $.
Cyclic shifts of the form $u_{p_i,j}^{<k>} $ when $k$ is a multiple of $ p_i $ can be written as a linear combination of 
elements in $\{ u_{p_i, \ell} \:|\: \ell =0,1, \ldots,j  \} $.
Hence we only need to consider the first $p_i$ shifts for each prime. 
We now show that $u_{p_i,j} $ along with their cyclic shifts form a generating set for the module of polynomial functions.

\begin{lemma}
\label{LemmaUi}
 $u_{p_i,j}$ defined in \eqref{upijs} is a polynomial function for  $\alli, \allj$.
\end{lemma}

\begin{proof}
 It suffices to provide a polynomial that evaluates to each of the functions. 
 For fixed $i \in \{ 1, \ldots t\}$ and  $ j \in \{0, \ldots, e_i-1 \}$ we give a polynomial that evaluates to 
 $u_{p_i,j}$. Consider the monomial $X ^{\phi(n)}$, where $\phi(n)$ is Euler's totient function. Since 
 $\phi(n) \geq e_i $ for $n > 1$, $a^{\phi(n)} \equiv 0$ (mod $p_i^{e_i}$) if $p_i \:| \: a$. If $p_i \nmid a$, $p_i$ 
 and $a$ are relatively prime and $ a^{\phi(n)} \equiv 1$ (mod $p_i^{e_i}$) by Euler's theorem. Hence for all $a \in \Z_n$ we have
 \[
 a^{\phi(n)} =
  \begin{cases}
   1 & (\mbox{mod } p_i^{e_i}) \quad \text{if } p_i \nmid a\\
   0 & (\mbox{mod } p_i^{e_i}) \quad \text{if } p_i\:|\: a.
  \end{cases}
\]

Then the polynomial $1-X^{\phi(n)} \equiv (n-1)X^{\phi(n)} + 1$ corresponds to function

\[
 (1-X^{\phi(n)})(a) =
  \begin{cases}
   1 & (\mbox{mod } p_i^{e_i}) \quad \text{if } p_i \:|\:a\\
   0 & (\mbox{mod } p_i^{e_i}) \quad \text{if } p_i \nmid a  
  \end{cases}
\]

and the polynomial $X^j(1-X^{\phi(n)})$ corresponds to the function
\[
 X^j(1-X^{\phi(n)})(a) =
  \begin{cases}
   a^j & (\mbox{mod } p_i^{e_i}) \quad \text{if } p_i \:|\:a \\
   0   & (\mbox{mod } p_i^{e_i}) \quad \text{if } p_i \nmid a 
  \end{cases}
\] for $ \allj$.

Since $\frac{\displaystyle n}{\displaystyle {p_i^{e_i}}}$ and $p_i^{e_i}$ are relatively prime we have 
\[
\frac{\displaystyle {n}}{\displaystyle{p_i^{e_i}}} X^j(1-X^{\phi(p^n)})(a) =
  \begin{cases}
  \frac{\displaystyle {n}}{\displaystyle{p_i^{e_i}}} a^j & \quad  \text{if } p_i \:|\:a \\
   0 & \quad \text{if } p_i \nmid a 
  \end{cases}
\]
which is the vector $u_{p_i,j}$ for $ \allj$.
\end{proof}

From Lemma \ref{cyclpolyn} it follows that the cyclic shifts $u_{p_i,j}^{<k>} $ are also polynomial functions 
for $k=0, \ldots, p_i-1$.

\begin{lemma}
 \{$u_{p_i,j}^{<k>} \:|\: \alli, \allj, k=0, \ldots, p_i-1 $\}  is $\Z_n$-multiplicatively-closed.
\end{lemma}

\begin{proof}
Case (i) : Consider $u_{p_i,j_1}^{<k_1>}$ and $u_{p_i,j_2}^{<k_2>}$ for a fixed $i$ where $k_1 \neq k_2$ and $j_1, j_2$ are arbitrary.
$$ (u_{p_i,j_1}^{<k_1>} \cdot u_{p_i,j_2}^{<k_2>})(a) = u_{p_i,j_1}^{<k_1>}(a)\cdot u_{p_i,j_2}^{<k_2>}(a) =0$$
since at least one of the two will be zero. \\
Case (ii) : For a fixed $i$ consider $u_{p_i,j_1}^{<k>}$ and $u_{p_i,j_2}^{<k>}$.
\begin{align*}
 (u_{p_i,j_1}^{<k>} \cdot u_{p_i,j_2}^{<k>})(a) &= u_{p_i,j_1}^{<k>}(a)\cdot u_{p_i,j_2}^{<k>}(a)\\
 &= \begin{cases}
  \frac{\displaystyle {n}}{\displaystyle{p_i^{e_i}}}(a-k)^{j_1} \cdot \frac{\displaystyle {n}}{\displaystyle{p_i^{e_i}}}(a-k)^{j_2} & \text{if } a \equiv k \; (\text{mod }p_i) \\
   0 & \mbox{ otherwise } 
  \end{cases}\\
 &= \begin{cases}    
\frac{\displaystyle n}{\displaystyle{p_i^{e_i}}}\frac{\displaystyle n}{\displaystyle{p_i^{e_i}}}(a-k)^{j_1 + j_2} &\text{ if } a \equiv k \; (\text{mod }p_i) \\
    0 & \text{ otherwise}
    \end{cases}\\
 &=\frac{\displaystyle n}{\displaystyle{p_i^{e_i}}}u_{p_i,j_1+j_2}^{<k>}.
 \end{align*}
Note that if $j_1 +j_2 \geq e_i$ then this corresponds to the zero function.\\
Case (iii) : Consider distinct $p_{i_1}$ and $p_{i_2}$ with arbitrary $j_1, j_2$. We need not consider cyclic shifts here 
since they are essentially same.
\begin{align*}
 (u_{p_{i_1},j_1} \cdot u_{p_{i_2},j_2})(a) &= u_{p_{i_1},j_1}(a)\cdot u_{p_{i_2},j_2}(a)\\
 &= \begin{cases}
   \frac{ \displaystyle n}{\displaystyle {p_{i_1}^{e_{i_1}}}}a^{j_1} 
   \cdot  \frac{\displaystyle n}{\displaystyle {p_{i_2}^{e_{i_2}}}}a^{j_2} & \text{if } p_{i_1}p_{i_2} \:|\:a  \\
   0 & \mbox{ otherwise. } 
  \end{cases}
 \end{align*}
But $n \:|\: \frac{ \displaystyle n}{ \displaystyle {p_{i_1}^{e_{i_1}}}} \frac{\displaystyle  n}{\displaystyle {p_{i_2}^{e_{i_2}}}}$, hence it is the zero function.
\end{proof}

Consider the sum of first $p_1$ shifts of $u_{p_1,0}$, $ \displaystyle\sum\limits_{k=0}^{p_1-1} u_{p_1,0}^{<k>}$.
This corresponds to the constant function $\frac{\displaystyle n}{\displaystyle p_1^{e_1}}(1,1, \ldots,1) $. One can similarly obtain the functions 
$\frac{\displaystyle n}{\displaystyle{p_i^{e_i}}}(1,1, \ldots,1) $ for $i=2, \ldots,t$. We know that as integers 
\begin{displaymath}
 \mbox{gcd}(\frac{n}{p_1^{e_1}},\frac{n}{p_2^{e_2}}, \ldots, \frac{n}{p_t^{e_t}} )=1.
\end{displaymath}
From Bezout's lemma there exist $a_1, a_2, \ldots, a_t \in \Z$ such that 
$$1 = a_1 \frac{n}{p_1^{e_1}} + a_2 \frac{n}{p_2^{e_2}} + \ldots + a_t \frac{n}{p_t^{e_t}}.$$
The above statement is true considering $\frac{\displaystyle n}{ \displaystyle p_i^{e_i}}, \alli$ as integers. 
 The statement is equally valid if we were to consider the corresponding equivalence classes modulo $n$.
Therefore we have:
\begin{displaymath}
(1,1,\ldots, 1)  = \displaystyle\sum\limits_{i=1}^{t} a_i\frac{\displaystyle {n}}{\displaystyle{p_i^{e_i}}}(1,1, \ldots,1)
=\displaystyle\sum\limits_{i=1}^{t} \displaystyle\sum\limits_{k=0}^{p_i-1} a_i u_{p_i,0}^{<k>}.
\end{displaymath}
This means the vector corresponding to the constant polynomial $1$ can be written as a linear combination of elements in 
$\{u_{p_i,0}^{<k>} \:|\: \alli, k=0, \ldots, p_i-1 \}$, \ie 
$ 1 \in \langle\langle u_{p_i,0}\: | \:\alli \rangle\rangle $. 
We will employ a similar method to show that vector corresponding to polynomial $X$, \ie $(0,1,2, \ldots,n-1)$ 
belongs to the module $ \langle \langle u_{p_i,j} \:|\: \alli, \allj   \rangle \rangle$.

\begin{lemma}
 Function induced by the polynomial $X$ can be written as a linear combination of elements in 
 $ \langle \langle u_{p_i,j} \:|\: \alli , \allj \rangle \rangle $.
 \end{lemma}
\begin{proof}
For a fixed $i$ consider the evaluation of  $\frac{\displaystyle n}{ \displaystyle {p_{i}^{e_i}}}X$.
 \begin{align*}
  \frac{n}{p_{i}^{e_i}}X &= \frac{n}{p_{i}^{e_i}}(0,1,2, \ldots,n-1) \\
  &= \frac{n}{p_{i}^{e_i}}(0,0,0, \ldots, p_i, 0, \ldots, 0, 2p_i, \ldots, 3p_i, \ldots, 0)\\
  & \quad + \frac{n}{p_{i}^{e_i}}(0,1 , 0, \ldots, 0,p_i + 1, \ldots, 2p_i +1, \ldots, 3p_i +1, \ldots, n-p_i+1, \ldots, 0)\\
  & \quad + \frac{n}{p_{i}^{e_i}}(0,0 , 2, \ldots, 0,p_i + 2, \ldots, 2p_i +2, \ldots, 3p_i +2, \ldots, n-p_i+2, \ldots, 0)\\
  & \qquad \qquad \qquad \qquad \qquad \qquad \vdots\\
  & \quad + \frac{n}{p_{i}^{e_i}}(0,\ldots, 0,p_i-1, \ldots, 2p_i -1, \ldots, 3p_i-1, \ldots, n-1)\\
  & = u_{p_i,1}\\
  & \quad + u_{p_i,1}^{<1>} + 1u_{p_i,0}^{<1>}\\
  & \quad + u_{p_i,1}^{<2>} + 2u_{p_i,0}^{<2>}\\
  & \qquad \qquad \vdots\\
  & \quad + u_{p_i,1}^{<p_i-1>} + (p_i-1)u_{p_i,0}^{<p_i-1>}.
 \end{align*}
Since we know that monomial $ X$ can be represented as a linear combination of 
$\frac{\displaystyle n}{ \displaystyle {p_{i}^{e_i}}}X$, it effectively means 
the function evaluated by polynomial $X$ belongs to the module generated by $u_{p_i,0}^{<k>}$ and $u_{p_i,1}^{<k>}$.
\end{proof}

Now that we have functions of $1, X \in \langle\langle u_{p_i,j} \rangle\rangle$, Lemma \ref{contlemma} directly gives us the following theorem.
\begin{theorem}
\label{MainResult}
 $ \{ u_{p_i,j}^{<k>} \:|\: \alli, \allj, k =0, \ldots, p_i-1 \}$ generates the $\Z_n$-module of polynomial function 
 from $\Z_n$ to itself.
 \end{theorem}

Written explicitly the generators are as follows.

\begin{align*}
u_{p_1,0} &= \frac{n}{p_{1}^{e_1}}(1, 0, \ldots, 0, 1, \ldots, 1, \ldots,0)\\
u_{p_1,1} &= \frac{n}{p_{1}^{e_1}}(0, 0, \ldots, 0, p_1, \ldots, 2p_1, \ldots,0)\\
& \qquad \qquad \qquad \vdots \\
u_{p_1, e_1-1} &= \frac{n}{p_{1}^{e_1}}(0, 0, \ldots, 0, p_1^{e_1-1}, \ldots, (2p_1)^{e_1-1}, \ldots,0)\\
& \qquad \qquad \qquad \vdots \\
u_{p_t, e_t-1} &= \frac{n}{p_{t}^{e_t}}(0, 0, \ldots, 0, p_t^{e_t-1}, \ldots, (2p_t)^{e_t-1}, \ldots,0).
\end{align*}

When $n$ is prime, $\Z_n $ is a field and the set of polynomial functions form a vector space 
over this field. The standard basis of the vector space is precisely the cyclic shifts of 
$$u_{n,0}=(1,\underbrace{0,\ldots,0}_\text{$n$-1 times}) $$ 
as mentioned in Section \ref{Intro}.

For the case when $n$ is a prime power of the form $p^e$ the generators are precisely those given in \cite{guhaalgorithmic}.
\begin{eqnarray*}
u_{p,0}  & = & (1,\underbrace{0,\ldots,0}_\text{$p-1$ times},1,\underbrace{0,\ldots,0}_\text{$p-1$ times}, 1,\ldots,0)\\
u_{p,1}  & = & (0,\underbrace{0,\ldots,0}_\text{$p-1$ times},p,\underbrace{0,\ldots,0}_\text{$p-1$ times},2p,\ldots,0)\\
u_{p,2}  & = & (0,\underbrace{0,\ldots,0}_\text{$p-1$ times},p^2,\underbrace{0,\ldots,0}_\text{$p-1$ times},(2p)^2,\ldots,0)\\
&  & \qquad \qquad \qquad \vdots \\
u_{p,e-1} & = &(0,\underbrace{0,\ldots,0}_\text{$p-1$ times},p^{e-1},\underbrace{0,\ldots,0}_\text{$p-1$ times},(2p)^{e-1},\ldots,0).
\end{eqnarray*}

When $n=p_1p_2 \ldots p_t$ the generators are cyclic shifts of 
\begin{eqnarray*}
u_{p_1,0}  & = & \frac{n}{p_1}(1,\underbrace{0,\ldots,0}_\text{$p_1-1$ times},1,\underbrace{0,\ldots,0}_\text{$p_1-1$ times}, 1,\ldots,0)\\
u_{p_2,0}  & = & \frac{n}{p_2}(1,\underbrace{0,\ldots,0}_\text{$p_2-1$ times},1,\underbrace{0,\ldots,0}_\text{$p_2-1$ times}, 1,\ldots,0)\\
&  & \qquad \qquad \qquad  \vdots \\
u_{p_t,0} & = & \frac{n}{p_t}(1,\underbrace{0,\ldots,0}_\text{$p_t-1$ times}, 1,\underbrace{0,\ldots,0}_\text{$p_t-1$ times},1,\ldots,0).
\end{eqnarray*}

\begin{example}
\label{eg12}
 Consider $n=12 = 2^2 \cdot 3 $. The generators are 
\begin{align*} 
 u_{2,0} &= (3,0,3,0,3,0,3,0,3,0,3,0)\\
 u_{2,1} &= (0,0,6,0,0,0,6,0,0,0,6,0)\\
 u_{3,0} &= (4,0,0,4,0,0,4,0,0,4,0,0)
\end{align*}
and their cyclic shifts. 
\end{example}

\section{Proposed Algorithm}
\label{MyAlgo}
\noindent
We present an algorithm based on Theorem \ref{MainResult} in Algorithm \ref{modifiedalgo}.
Let $N$ be the number of generators given by 
Theorem \ref{MainResult}. For each prime $p_i$, there are $e_i$ generators and $p_i$ cyclic shifts for each of them, which gives
\begin{equation}
\label{noofgen}
 N=p_1e_1 + p_2e_2 + \ldots +p_te_t. 
\end{equation}
Let $u_0, u_1, \ldots, u_N $ be the generators and the input function 
$f: \Z_n \longrightarrow \Z_n$ be of the form $(b_0, b_1, \ldots, b_{n-1}) \in \Z_n^n$.
The key idea is to check if the given function is a linear combination of 
these generators. 
This is equivalent to checking if the following system of linear equations in variables $y_1, \ldots, y_{N} $ has a solution in $\Z_n$.
\begin{equation}
\label{fullmatrix}
 A \cdot 
\left(
\begin{array}{c}
 y_1 \\
 y_2 \\
  \vdots \\
 y_{N}
\end{array}
\right)
=
\left(
\begin{array}{c}
 b_0 \\
 b_1 \\
 \vdots \\
 b_{n-1}
\end{array}
\right)
\end{equation}
where $A$ is an $n \times N $ matrix whose columns are the vectors $u_1, u_2, \ldots, u_N$. 
Since $N < n$, we have an over-defined system of equations. It is efficient to truncate the $ n \times N$ matrix $A$ and 
solve for the first $N$ rows using Gaussian elimination to determine if a solution exists. Let $B$ be the $ N \times N$ matrix which 
consists of the first $N$ rows of $A$ in \eqref{fullmatrix}. We now solve for the smaller matrix 
\begin{equation}
\label{reducedmatrix}
 B \cdot 
\left(
\begin{array}{c}
 y_1 \\
 y_2 \\
  \vdots \\
 y_{N}
\end{array}
\right)
=
\left(
\begin{array}{c}
 b_0 \\
 b_1 \\
 \vdots \\
 b_{N-1}
\end{array}
\right)
\end{equation} 
which is equivalent to checking for the first $N$ components of $f$ and then verify if the solution holds for remaining 
components of $f$.

\begin{algorithm}
\caption{Determination of Polynomial Functions}
\begin{algorithmic}

\State \textbf{Input:} $f=(b_0,b_1,\ldots,b_{n-1})$, where $n= p_1^{e_1} \ldots p_t^{e_t}  $.

\State \For {$i=1, \ldots,  t$} \Comment Step 1
	    \For{$j=0, \ldots, p_i-1 $}
		\For{$\ell =1, \ldots, \frac{n}{p_i}-1 $}
		    \If {$b_j \not\equiv b_{j+ \ell p_i} \; (\mbox{mod } p_i)$}
			\State \textbf{Output:} $f$ is not polynomial.
			\State exit
		    \EndIf
		\EndFor
	    \EndFor
       \EndFor

\Comment Step 2
\If{$B
\left( 
\begin{array}{c}
y_1\\
\vdots\\
y_{N}            
\end{array}
\right) = 
\left( 
\begin{array}{c}
b_0\\
\vdots\\
b_{N-1}            
\end{array}
\right)$ has no solution} \Comment B as in ~\eqref{reducedmatrix}

	\State \textbf{Output:} $f$ is not polynomial.
	\State exit
\EndIf

\State Let $(d_1, d_2, \ldots, d_N) $ be the solution.

\For{$j=N , \ldots, n-1 $} 		\Comment Step 3 
  \If{$b_j \neq ( \displaystyle\sum\limits_{i=1}^{N} d_i u_i)(j)$}
	    \State \textbf{Output:} $f$ is not polynomial.
	    \State exit
      \Else
	    \State \textbf{Output:} $f$ is polynomial.
      \EndIf    

\EndFor
\end{algorithmic}
\label{modifiedalgo}
\end{algorithm}

\subsection*{Proof of correctness}

Step 1 of Algorithm \ref{modifiedalgo} checks for a necessary congruence condition every polynomial must satisfy. Step 1 checks if
\begin{equation}
\label{congr} 
g(a) \equiv g(a+ p_i \ell) \; (\mbox{mod }p_i)
\end{equation}
for all $p_i \: | \: n$ and $a , \ell \in \Z_n.$ This step is useful in identifying a significant fraction of non-polynomial functions.
Assuming that each entry in the $n$-tuple is arbitrary, for any prime factor $p_i$, number of functions that satisfy  
\eqref{congr} is $ n^{p_i} \left( \frac{n}{p_i} \right)^ {n-p_i}$. In other words the fraction of functions that satisfy the 
condition is $\frac{1}{p_i^{n-p_i}}$ for each $p_i$, where $\alli$. Since the number of polynomial functions is much smaller than 
the number of non-polynomial functions, for most input functions the algorithm terminates at this step. 

Step 2 computes the solution of \eqref{reducedmatrix}. One must bear in mind that all 
computations are performed modulo $n$ where division by $p_i$ is not defined for $ \alli$. This means that whenever we encounter 
a case where division by $p_i$ occurs it immediately implies that no solution exists in $\Z_n$, and the function $f$ 
is not polynomial. 

Suppose a solution exists, say, $(d_1, d_2, \ldots, d_N)$. Step 3 checks if the solution holds for all components, \ie if
$$f= d_1u_1 + d_2u_2 + \ldots d_Nu_N. $$

The following example will help us understand the correctness of the algorithm.
\begin{example}
 Consider $n=12$. Let $f =(0,1,4,9,4,1,0,1,4,9,4,1)$.
\end{example}
The generators are given in Example \ref{eg12}. 
Step 1 checks if 
$$f(x+2) \equiv f(x) \mbox{ (mod } 2),$$
$$f(x+3) \equiv f(x) \mbox{ (mod } 3),$$
which is satisfied by $f$. Step 2 checks if the following system has a solution.
\[
\left(
 \begin{tabular}{ccccccc}
 4  & 0 & 0 & 3 & 0 & 0 & 0 \\ 
 0  & 4 & 0 & 0 & 3 & 0 & 0 \\
 0  & 0 & 4 & 3 & 0 & 6 & 0 \\
 4  & 0 & 0 & 0 & 3 & 0 & 6 \\
 0  & 4 & 0 & 3 & 0 & 0 & 0 \\
 0  & 0 & 4 & 0 & 3 & 0 & 0 \\
 4  & 0 & 0 & 3 & 0 & 6 & 0 \\
\end{tabular}
\right) \cdot
\left(
\begin{array}{c}
 y_1 \\
 y_2 \\
 y_3 \\
 y_4 \\
 y_5 \\
 y_6 \\
 y_7
\end{array}
\right) = 
\left(
\begin{array}{c}
 0 \\
 1 \\
 4 \\
 9 \\
 4 \\
 1 \\
 0
\end{array}
\right)
\]
A solution exists, namely $(0,1,1,0,3,0,0)$. In step 3 we check if the solution satisfies for all components, which it does.
Hence $f$ is polynomial.

\section{Analysis of Algorithm}
\label{Algocomplexity}
\noindent
In this section we discuss the computational aspects of the algorithm. We study the complexity of Algorithm \ref{modifiedalgo} and 
compare the performance of the algorithm with 
one based on the canonical set of generators, which is the set of functions corresponding to the monomials 
$ \{ 1, X, X^2, \ldots, X^{ \mu -1}\}$, where $ \mu$ is as defined earlier. 

First let us consider the set of generators $ \{ 1, X, X^2, \ldots , X^{\mu -1}\}$. 
Let $\Lambda = \underset{1 \leq i \leq t} \max \{ p_i e_i\}$. 
We have $ \mu \leq \Lambda$, 
\ie $ \Lambda $ is an upper bound of $\mu $ for a given $n$, since $p_i^{e_i}$ divides $(p_i e_i)! $ for $\alli$ . 
In order to determine if the given function 
$f = (b_0, b_1, \ldots, b_{n-1})$ is polynomial we need to check if there exist scalars $c_0, c_1, \ldots, c_{\mu-1} \in \Z_n$ 
such that $c_0 + c_1 X + \ldots + c_{\mu -1}X^{\mu -1} $ evaluates to $f$.  This is equivalent to determining 
if the following system of linear equations similar to \eqref{fullmatrix} has a solution.
\begin{equation}
\label{canonicalmatrix} 
 \left(
 \begin{array}{ccccc}
 1 & 0 & 0    & \ldots & 0 \\
 1 & 1 & 1^2  & \ldots & 1^{\mu -1} \\
 1 & 2 & 2^2  & \ldots & 2^{\mu -1} \\
 \vdots & \vdots & \vdots & \ddots & \vdots \\
 1 & (n-1) & (n-1)^2 & \ldots & (n-1)^{\mu -1}
 \end{array}
 \right)
\left(
\begin{array}{c}
 y_0 \\
 y_1 \\
 y_2 \\
 \vdots \\
 y_{\mu-1}
\end{array}
\right)
=
\left(
\begin{array}{c}
 b_0 \\
 b_1 \\
 b_2 \\
 \vdots \\
 b_{n-1}
\end{array}
\right)
\end{equation}

It is simpler to consider the first $\mu$ rows like we do in Algorithm \ref{modifiedalgo}. 
This step involves $O(\mu ^3)$ operations. Suppose a solution exists, say, $(c_0, c_1, \ldots, c_{\mu-1}) $. We then check 
if $ \sum_{j=0}^{\mu -1} c_{j}i^j = b_i$, for $i= \mu, \ldots, n-1 $. Evaluating the polynomial 
$\sum_{j=0}^{\mu -1} c_{j}X^j $ at each $i$ requires $\mu$ multiplications. To evaluate the polynomial 
at $(n-\mu)$ points takes $O(n\mu) $ steps. Comparing the polynomial with $(b_0, b_1, \ldots, b_{n-1})$ requires one pass 
which is $n$ operations.

Hence the complexity is $ O(\mu^3) + O(n \mu) + O(n)$. For small values of $n$ the term $\mu ^3$ dominates. However, 
as $n$ increases $\mu \ll n$ and the effective complexity is $ O(n \mu)$. Depending on the factorization of $n$, at worst it can 
be $O(n \Lambda)$.

We claim that the new set of generators reduces the number of computations. 
Before we proceed to study the complexity of the algorithm based on new generators we must remember that the complexity does not depend solely on the 
size of the input $n$, but also on the factorization of $n$.

Our proposed algorithm contains steps identical to the method mentioned above. Let $N$ be the number of generators from 
\eqref{noofgen}. Observe that $\Lambda \leq N \leq \Lambda t$.

Step 1 requires a single pass for each prime factor $p_i$ of $n$ and takes $O(n t)$, 
since each traversal of the input takes $O(n)$ and there are $t$ such primes. This does not affect the overall complexity 
adversely. In practice, it makes the algorithm faster by identifying many non-polynomial functions early without resorting to 
matrix calculations.

Step 2 solves for \eqref{reducedmatrix}. We have an $N \times N$ matrix $B$ which is larger than the one considered in 
\eqref{canonicalmatrix} at most by factor of $t$. An important feature of matrix $B$ is that most of its entries are zero. 
We can show that each row of $A$ in \eqref{fullmatrix} contains at most $(e_1 + \ldots + e_t)$ non-zero entries. 
Hence, of the $nN$ entries of $A$ at most $n(e_1 + \ldots + e_t)$ are non-zero. Moreover, the non-zero values are distributed 
uniformly over $A$. In other words,
\begin{align*}
\mbox{Fraction of non-zero entries}
& < \frac{n(e_1 + e_2 + \ldots + e_t)}{nN} \\
& < \frac{e_1 + e_2 + \ldots + e_t}{p_1e_1 + p_2e_2 + \ldots + p_te_t},
\end{align*}
which is also the fraction of non-zero entries in matrix $B$. 
This step takes $O(N^3)$ to solve in the worst case. It is possible that the matrix may be solved faster if it is sufficiently sparse.  
Suppose $(d_1 , \ldots, d_N) $ is a solution. 

Step 3 checks if $\sum_{i=1}^{N}d_i u_i(j) = b_j$ for the remaining components $j= N , \ldots, n-1$. Since each row contains 
fewer non-zero entries this step takes at most $n(e_1 + e_2 + \ldots + e_t) $ multiplications compared to $n \mu$ earlier. 
We can determine if  
$ \sum_{i=1}^{N}d_i u_i = f $ in a single pass that takes $O(n)$. The total complexity of the new algorithm 
is $O(nt + N^3 + n(e_1 + \ldots + e_t)) $. Although the second term is comparable to $\Lambda^3 t^3 $, for large values of $n$ 
the algorithm with new generators requires fewer multiplications.

Step-by-step break up of the complexity is given below.
\begin{eqnarray*}
\mbox{T}(n) 
& = & O(nt) + O(N^3) + O(n(e_1 + \ldots + e_t))\\
& = & O(n (e_1 + \ldots + e_t))
\end{eqnarray*}

The total complexity varies in magnitude depending on the factorization of $n$. We list some special cases where the 
algorithm performs significantly better.\\
Case 1: $n=p$. Algorithm \ref{modifiedalgo} takes $O(n)$. Note that in this case $N=n$, hence Step 2 
takes $O(n^3)$. However, this computation is unnecessary, since we know a priori every function 
is polynomially representable. \\
Case 2: $n=p^e$. Algorithm \ref{modifiedalgo} takes $O(ne)$.\\
Case 3: $n=p_1p_2 \ldots p_t$. Algorithm \ref{modifiedalgo} takes $O(nt)$.

Some examples are given below. 
\begin{example}
Consider $n= 29 \cdot 37^3 \cdot 53$. Then 
$ \mu = 111$,  $N=193$, $ \sum e_i = 5$. 
This means the new algorithm solves for a matrix roughly twice as big. However, the number of multiplications in Step 3 of 
our algorithm is $5n$ compared to $111n$ that may be required with  canonical generators.
\end{example}
\begin{example}
 Consider $n=97 \cdot 101 .$ Then
$ \mu =101$ , $N =198$, $ \sum e_i = 2$. 
Once again the system of equations is larger, 
but the multiplications required is comparable to $2n$ instead of $101n$.
\end{example}
The above example represents the case where our algorithm performs much better, when the number of prime factors and the exponent of each prime is small. Note that the matrix required in Step 2 of Algorithm \ref{modifiedalgo} is highly sparse, containing only two entries per row. This drastically reduces the calculation to solve the system of equations. 
The new algorithm performs poorly in the case when the exponents are much larger compared to the prime factors, in particular when the exponent $e_i$ is much greater than $p_i^2$.
\begin{example}
Consider $ n= 2^{15} \cdot 3^{10} \cdot 5^6.$ Then 
$ \mu = 25$, $ N =90$, $\sum e_i = 31$ . 
The system of equations is larger yet the number of multiplications required turns out to be $31n$ rather than $25n$. Although in this case the number of multiplications is more, it still remains comparable to that using canonical generators.
\end{example}

One must note that in the above case $\mu$ was less than $ \Lambda$, whereas in the previous examples the equality held. 
It is reasonable to assume that this equality holds for most values of $n$. Indeed, for a fixed positive integer $n$ within a 
large enough upper bound, the possible values $e_i$ may take is much less compared to those of $p_i$. Hence the new 
characterization is more efficient to check for polynomial representability for most values of $n$. 

The next question that follows is determining the polynomial that evaluates to the given function. This is possible since we have the polynomials that correspond to the generators from Lemma \ref{LemmaUi} and the 
algorithm gives a suitable linear combination of generators. The polynomial thus obtained has a degree of $ \phi(n)$. It is possible to get a lower degree polynomial  by simply dividing it by $X(X-1)\ldots (X-\mu+1)$. The remainder is of degree less than $ \mu$ and evaluates to the same function. By similarly choosing suitable coefficients it is possible to arrive at the canonical representation mentioned in 
\cite{singmaster1974polynomial}.
\section{Polynomials in several variables}
\label{Multivar}
\noindent
The set of generators described so far can be extended to multivariate functions in a natural way. 
Consider the set of functions in $m$ variables $x_1, x_2, \ldots, x_m$ over $\Z_n$. The intuitive set of generators to represent the polynomial functions are the evaluations of the monomials $X_1^{\alpha_1} X_2^{\alpha _2} \cdots X_m^{\alpha_m}$, where $\alpha_i = 0, \ldots, \mu-1$, for $ i=1, \ldots, m$, leading to $\mu^m$ generators. 
We wish to give a set of generators similar to that given in Theorem \ref{MainResult} for 
the set of polynomial functions from $ \Z_n^m$ to $\Z_n$. 

\begin{proposition}
\label{resultformultivar}
 The module of polynomial functions in $m$-variables from $\Z_n^m$ to $\Z_n$ is generated by tensor product of vectors given for $\Z_n$ taken 
 $m$ at a time and their shifts, \ie
generators are given by $u_{p_{i_1},j_1} \otimes u_{p_{i_2},j_2} \otimes \ldots \otimes u_{p_{i_m},j_m} $, where
$$(u_{p_{i_1},j_1} \otimes \ldots \otimes u_{p_{i_m},j_m})(a_1, \ldots, a_m) = u_{p_{i_1},j_1}(a_1)\ldots u_{p_{i_m},j_m}(a_m). $$
 \end{proposition}

It must be noted that if $p_{i_1} \neq p_{i_2}$ (or any other pair), then the tensor product is simply zero. Effectively the generators are of the form 
$u_{p_{i},j_1} \otimes u_{p_{i},j_2} \otimes \ldots \otimes u_{p_{i},j_m}$, where $ \alli$.
For a fixed $p_i$ the number of generators, ignoring the shifts, is the number of solutions to the inequality $j_1 + j_2 + \ldots + j_m < e_i $ 
which is $\binom{m+ e_i -1}{m} $. 
For each of these tensors there are $p_i$ shifts along each of the $m$ dimensions. Hence the number of generators 
corresponding to each $p_i$ is $p_i^{m}\binom{m+e_i-1}{m}$. 
Summing up over all primes we get the total number of generators to be
$$N_m = p_1^m\binom{m+e_1 -1}{m}+ \ldots+ p_t^m\binom{m+e_t-1}{m}. $$ 
Note that when we substitute $m=1$, \ie the univariate case we get precisely 
$p_1e_1+ \ldots + p_te_t$ from \eqref{noofgen} mentioned in Section \ref{MyAlgo}.
\begin{example}
Consider the case of functions in two variables over $\Z_6$ . The generators are
\[
\begin{tabular}{l|cccccc}
         & 2 & 0 & 0 & 2 & 0 & 0 \\ \hline
 2       & 4 & 0 & 0 & 4 & 0 & 0 \\
 0       & 0 & 0 & 0 & 0 & 0 & 0 \\
 0       & 0 & 0 & 0 & 0 & 0 & 0 \\
 2       & 4 & 0 & 0 & 4 & 0 & 0 \\
 0       & 0 & 0 & 0 & 0 & 0 & 0 \\
 0       & 0 & 0 & 0 & 0 & 0 & 0 \\
\end{tabular}
\]

\[
\begin{tabular}{l|cccccc}
         & 3 & 0 & 3 & 0 & 3 & 0 \\ \hline
 3       & 3 & 0 & 3 & 0 & 3 & 0 \\
 0       & 0 & 0 & 0 & 0 & 0 & 0 \\
 3       & 3 & 0 & 3 & 0 & 3 & 0 \\
 0       & 0 & 0 & 0 & 0 & 0 & 0 \\
 3       & 3 & 0 & 3 & 0 & 3 & 0 \\
 0       & 0 & 0 & 0 & 0 & 0 & 0 \\
\end{tabular}
\]
Note that $(2,0,0,2,0,0) \otimes (3,0,3,0,3,0)$ is a zero-matrix.
\end{example}

One may proceed to give an algorithm similar to Algorithm \ref{modifiedalgo} for polynomials in several variables. 
Let $ f : \Z_n^m \longrightarrow \Z_n $ be a function in $m$ variables. We may now represent the function 
as an $n^m$-tuple $(b_0, b_1, \ldots, b_{n^m-1})$. Let $u_1, \ldots, u_{N_m}$ be the generators, 
each represented as an $n^m$-tuple. We check if the given function $f$ is a linear combination of the generators, 
leading to a system of linear equations in variables $y_1, \ldots, y_{N_m} $, similar 
to  \eqref{fullmatrix} in Section \ref{MyAlgo}.
\begin{equation}
\label{fullmatrixmultivar}
 A' \cdot 
\left(
\begin{array}{c}
 y_1 \\
 y_2 \\
  \vdots \\
 y_{N_m}
\end{array}
\right)
=
\left(
\begin{array}{c}
 b_0 \\
 b_1 \\
 \vdots \\
 b_{n^m-1}
\end{array}
\right)
\end{equation}
where $A'$ is an $n^m \times N_m$ matrix whose columns are the generators 
$u_1, \ldots, u_{N^m}$. We solve for a smaller $B'$ consisting of the first $N_m$ 
rows of $A'$ and then verify whether the solution holds for the remaining components.

\begin{equation}
\label{reducedmatrixmultivar}
 B' \cdot 
\left(
\begin{array}{c}
 y_1 \\
 y_2 \\
  \vdots \\
 y_{N_m}
\end{array}
\right)
=
\left(
\begin{array}{c}
 b_0 \\
 b_1 \\
 \vdots \\
 b_{N_m-1}
\end{array}
\right)
\end{equation} 

A sketch of algorithm is given in Algorithm \ref{modifiedmultialgo}.

\begin{algorithm}
\caption{Determination of Polynomial Functions in several variables}
\begin{algorithmic}
\State		\Comment Step 1
\If{$B'
\left( 
\begin{array}{c}
y_1\\
\vdots\\
y_{N_m}            
\end{array}
\right) = 
\left( 
\begin{array}{c}
b_0\\
\vdots\\
b_{N_m-1}            
\end{array}
\right)$ has no solution}

	\State \textbf{Output:} $f$ is not polynomial. \Comment $B'$ as in ~\eqref{reducedmatrixmultivar}
	\State exit
\EndIf

\State Let $(d_1, d_2, \ldots, d_{N_m}) $ be the solution.

\State 
\For{$j=N_m , \ldots, n^m-1 $} \Comment Step 2
  \If{$b_j \neq ( \displaystyle\sum\limits_{i=1}^{N_m} d_i u_i)(j)$}
	    \State \textbf{Output:} $f$ is not polynomial.
	    \State exit
      \Else
	    \State \textbf{Output:} $f$ is polynomial.
      \EndIf    

\EndFor
\end{algorithmic}
\label{modifiedmultialgo}
\end{algorithm}

As a preliminary step one can check if 
$$f(a_1, \ldots, a_{m}) \equiv f(a_1 + \ell_1 p_i,  \ldots, a_m + \ell_m p_i) \mbox{ mod }p_i,  $$
where $a_1, \ldots, a_m  \in \{ 0, 1, \ldots,p_i -1 \}$ and  
$\ell_1, \ldots, \ell_m  \in  \{ 1,  \ldots, \frac{n}{p_i} -1 \}$, for each $p_i$, 
as in Step 1 of Algorithm \ref{modifiedalgo}.

Note that in \eqref{fullmatrixmultivar} the number of rows in $A'$ is the size of input 
as in the case of single variable, but the matrix is much sparser. For each generator 
in $\{ u_j \:|\: j=1, \ldots, N_m\}$, if $u_j$ corresponds to some prime $p_i$ then 
the fraction of non-zero entries in $u_j$ is at most $\frac{1}{\displaystyle{p_i^m}}$. Hence 
each prime $p_i$ contributes 
$(\frac{\displaystyle n}{\displaystyle p_i})^m \binom{m+e_i-1}{m} $ non-zero entries to the matrix 
$A'$.
For the matrix $A'$ in 
\eqref{fullmatrixmultivar} 

\begin{align*}
\mbox{Fraction of non-zero entries}
& < \frac{n^m(\displaystyle\sum_{i=1}^{t}\binom{m+e_i-1}{m})}{n^m N_m} \\
& < \frac{\displaystyle\sum_{i=1}^{t}\binom{m+e_i-1}{m}}{ \displaystyle\sum_{i=1}^{t}p_i^{m}\binom{m+e_i-1}{m}}.
\end{align*}

The uniform distribution of non-zero entries ensures that matrix $B'$ in Step 1 of Algorithm \ref{modifiedmultialgo} also has same sparsity. Assuming the system of linear equations is solved using Gaussian elimination, Step 1 of Algorithm \ref{modifiedmultialgo} takes $O(N_m^3)$. 
Step 2 takes $O(n^m \displaystyle\sum_{i=1}^{t}\binom{m+e_i-1}{m})$ multiplications since 
each row of $A'$ (and $B'$) contains at most $\displaystyle\sum_{i=1}^{t}\binom{m+e_i-1}{m}$ non-zero entries. The total 
time complexity  of Algorithm \ref{modifiedmultialgo} is given by 
$$T(n,m) = O(N_m^3) + O(n^m \displaystyle\sum_{i=1}^{t}\binom{m+e_i-1}{m}).$$
For most $n$ and $m$, the values of $N_m$ and $\sum_{i=1}^{t}\binom{m+e_i-1}{m}$ 
are smaller than $n^m$ by several orders of magnitude. Hence the time complexity of the 
algorithm  for functions in several variables is much less compared to the methods that
result out of canonical generators or the characterization given in 
\cite{hungerbuhler2006generalization} .

\section{Conclusion} 
\label{coda} 
\noindent 
In this paper we have provided an alternate characterization of polynomial functions over  
$\Z_n$ that results in improved algorithms for deciding polynomial representability. 
This characterization makes use of module structure of the set of polynomial functions and 
can be used to give a polynomial that evaluates to the polynomially representable function. By this, we can also arrive at the canonical representation  mentioned in \cite{singmaster1974polynomial}.
In addition, the characterization is also extended to polynomial
functions in several variables.

\footnotesize{
\section*{Acknowledgments}
Authors would like to thank anonymous reviewers for bringing our
attention to the work of Hungerb\"uhler and Specker (2006) and
providing valuable suggestions that improved this paper immensely.

}
\end{document}